\documentclass[a4paper]{article}
\usepackage{enumerate}
\usepackage{latexsym}
\usepackage{amssymb}
\usepackage{amsthm}
\usepackage{amsmath}
\usepackage[mathscr]{eucal}

\theoremstyle{plain}
\newtheorem{thm}{Theorem}[section]
\newtheorem{lem}[thm]{Lemma}
\newtheorem{prop}[thm]{Proposition}
\newtheorem{cor}[thm]{Corollary}

\theoremstyle{definition}
\newtheorem{exmp}[thm]{Example}
\newtheorem{rem}[thm]{Remark}
\newtheorem{defn}[thm]{Definition}

\newtheorem{prob}[thm]{Problem}

\DeclareMathOperator{\Sol}{Sol}

\begin{document}
\title{On a generalization of distance sets}

\author{Hiroshi Nozaki\\ 
Graduate School of Mathematics 
Kyushu University,\\
nozaki@math.kyushu-u.ac.jp\\
and\\
Masashi Shinohara\\
Suzuka National College of Technology\\
shinohara@genl.suzuka-ct.ac.jp 
}

\maketitle

\begin{center}
\textbf{Abstract}
\end{center}

A subset $X$ in the $d$-dimensional Euclidean space is called a $k$-distance set 
if there are exactly $k$ distinct distances between two distinct points in $X$ 
and a subset $X$ is called a locally $k$-distance set if 
for any point $x$ in $X$, there are at most $k$ distinct distances between $x$ and other points in $X$. 

Delsarte, Goethals, and Seidel gave the Fisher type upper bound for the cardinalities of $k$-distance sets on a sphere in 1977. 
In the same way, we are able to give the same bound for locally $k$-distance sets on a sphere. 
In the first part of this paper, we prove that if $X$ is a locally $k$-distance set attaining the Fisher type upper bound, then determining a weight function $w$, $(X,w)$ is a tight weighted spherical $2k$-design. This result implies that locally $k$-distance sets attaining the Fisher type upper bound are $k$-distance sets. 
In the second part, we give a new absolute bound for the cardinalities of $k$-distance sets on a sphere.  
This upper bound is useful for $k$-distance sets for which the linear programming bound is not applicable.  
In the third part, we discuss about locally two-distance sets in Euclidean spaces. 
We give an upper bound for the cardinalities of locally two-distance sets in Euclidean spaces. 
Moreover, we prove that the existence of a spherical two-distance set in $(d-1)$-space which attains the Fisher type upper bound is equivalent to 
the existence of a locally two-distance set but not a two-distance set in $d$-space with more than $d(d+1)/2$ points. 
We also classify optimal (largest possible) locally two-distance sets for dimensions less than eight. 
In addition, we determine the maximum cardinalities of locally two-distance sets on a sphere for dimensions less than forty. 

\section{Introduction}
Let $\mathbb{R}^d$ be the $d$-dimensional Euclidean space. 
For $X\subset \mathbb{R}^d$, 
let $A(X)=\{d(x,y)|x, y\in X,x\ne y\}$ where $d(x,y)$ is the Euclidean distance 
between $x$ and $y$ in $\mathbb{R}^d$. 
We call $X$ a {\it $k$-distance set} if $|A(X)|=k$. 
Moreover for any $x\in X$, define $A_X(x)=\{d(x,y)|y\in X,x\ne y\}$. 
We will abbreviate $A(x)=A_X(x)$ whenever there is no risk of confusion.
A subset $X\subset \mathbb{R}^d$ is called a {\it locally $k$-distance set}  
if $|A_X(x)|\leq k$ for all $x\in X$. 
Clearly every $k$-distance set is a locally $k$-distance set. 
A locally $k$-distance set is said to be {\it proper} if it is not a $k$-distance set. 
Two subsets in $\mathbb{R}^d$ are said to be isomorphic if there exists a similar transformation from one to the other. 
An interesting problem for $k$-distance sets (resp.\ locally $k$-distance set) is to determine 
the largest possible cardinality of $k$-distance sets (resp.\ locally $k$-distance set) in $\mathbb{R}^d$. 
We denote this number by $DS_d(k)$ (resp.\ $LDS_d(k)$) and a $k$-distance set $X$ (resp.\ locally $k$-distance set $X$) in $\mathbb{R}^d$ 
is said to be {\it optimal} if $|X|=DS_d(k)$ (resp.\ $LDS_d(k)$). 
Moreover we denote the maximum cardinality of a $k$-distance set (resp.\ locally $k$-distance set) in the unit sphere $S^{d-1} \subset \mathbb{R}^d$ by $DS^{\ast}_d(k)$ (resp.\ $LDS^{\ast}_d(k)$). 

For upper bounds on the cardinalities of distance sets in $\mathbb{R}^d$, Bannai-Bannai-Stanton \cite{Ban2} 
and Blokhuis \cite{Blo} gave $DS_d(k)\leq \binom{d+k}{k}$. For $k=2$, the numbers $DS_d(2)$ are known for $d\leq 8$ 
(Kelly \cite{Kelly}, Croft \cite{Croft} and Lison\v{e}k \cite{Liso}). 
For $d=2$, the numbers $DS_2(k)$ are known and optimal $k$-distance sets are classified for $k\leq 5$ 
(Erd\H{o}s-Fishburn \cite{Erd2}, Shinohara \cite{Shino1}, \cite{Shino2}). 
Moreover we have $DS_3(3)=12$ and every optimal three-distance set is isomorphic to the set of 
vertices of a regular icosahedron (Shinohara \cite{Shino3}). 

\begin{center}
$
\begin{array}{c|cccccccc}
d & 1 & 2 & 3 & 4 & 5 & 6 & 7 & 8\\
\hline
DS_d(2) & 3 & 5 & 6 & 10 & 16 & 27 & 29 & 45\\
\end{array}
$
\quad 
$
\begin{array}{c|cccccccc}
k & 1 & 2 & 3 & 4 & 5 \\
\hline
DS_2(k) & 3 & 5 & 7 & 9 & 12\\
\end{array}
$

Table: Maximum cardinalities for two-distance sets and planar $k$-distance sets
\end{center} 

We have a lower bound for $DS^{\ast}_d(2)$ of $d(d+1)/2$ since the set of all midpoints of the edges of a $d$-dimensional regular simplex is a two-distance set on a sphere with $d(d+1)/2$ points. Musin determined that $DS^{\ast}_d(2)=d(d+1)/2$ for $7\leq d \leq 21$, $24\leq d \leq 39$ \cite{Musin}. For $2\leq d \leq 6 $, we have $DS^{\ast}_d(2)=DS_d(2)$ and for $d=22$, we have $DS^{\ast}_d(2)=275$. For $d=23$, $DS^{\ast}_d(2)=276$ or $277$ \cite{Musin}.  

Delsarte, Goethals, and Seidel gave the Fisher type upper bound for the cardinalities of $k$-distance sets on a sphere \cite{Del}. This upper bound also applies to locally $k$-distance sets on a sphere. 
\begin{thm}[Fisher type inequality \cite{Del}]
$(\mathrm{i})$ Let $X$ be a locally $k$-distance set on $S^{d-1}$. Then, $|X| \leq \binom{d+k-1}{k}+\binom{d+k-2}{k-1}(=:N_d(k))$. \\
$(\mathrm{ii})$ Let $X$ be an antipodal ({\it i.e.\ }for any $x\in X$, $-x \in X$) locally $k$-distance set on $S^{d-1}$. Then, $|X| \leq 2 \binom{d+k-2}{k-1}(=:N^{\prime}_d(k))$.
\end{thm}
It is well known that if a $k$-distance set $X$ attains this upper bound, then $X$ is a tight spherical design. We will give the definition of spherical designs in the next section. 
Of course, $k$-distance sets which attain this upper bound are optimal. This optimal $k$-distance set is very interesting because of its relationship with the design theory. 
Classification of tight spherical $t$-designs have been well studied in \cite{Bannai-Damerell1,Bannai-Damerell2,Bannai-Munemasa-Venkov}. 
Classifications of tight spherical $t$-designs are complete, except for $t=4,5,7$. 
This implies that classifications of $k$-distance sets (resp.\ antipodal $k$-distance sets) which attain this upper bound are complete,  
except for $k=2$ (resp.\ $k=3,4$).  
For $t=4$, a tight spherical four-design in $S^{d-1}$ exists only if $d=2$ or $d=(2l+1)^2-3$ for a positive integer $l$ and  
the existence of a tight spherical four-design in $S^{d-1}$ is known only for $d=2, 6$ or $22$.  \\

In Section \ref{sphere}, we prove the following theorem. 

\begin{thm} \label{weighted-locally}
$(\mathrm{i})$ Let $X$ be a locally $k$-distance set on $S^{d-1}$. If $|X|= N_d(k)$, then for some determined weight function $w$, $(X,w)$ is a tight weighted spherical $2k$-design. 
Conversely, if $(X,w)$ is
a tight weighted spherical $2k$-design, then $X$ is a locally $k$-distance set (indeed, $X$ is
a $k$-distance set). \\
$(\mathrm{ii})$ Let $X$ be an antipodal locally $k$-distance set on $S^{d-1}$. If $|X|= N_d^{\prime}(k)$, then for some determined weight function $w$, $(X,w)$ is a tight weighted spherical $(2k-1)$-design.  Conversely, if $(X,w)$ is a
tight weighted spherical $(2k - 1)$-design, then $X$ is an antipodal locally $k$-distance set
(indeed, $X$ is an antipodal $k$-distance set).
\end{thm}

This theorem implies that the concept of locally distance sets is a natural generalization of distance sets, because this theorem is a generalization of the relationship between tight spherical designs and distance sets.

Indeed, Theorem \ref{weighted-locally} implies the following. 

\begin{thm}\label{main of sphere}
$(\mathrm{i})$ Let $X$ be a locally $k$-distance set on $S^{d-1}$. If $|X|= N_d(k)$, then $X$ is a $k$ -distance set. \\
$(\mathrm{ii})$ Let $X$ be an antipodal locally $k$-distance set on $S^{d-1}$. If $|X|= N_d^{\prime}(k)$, then $X$ is a $k$-distance set. 
\end{thm}

In Section \ref{section bound}, we give a new upper bound for $k$-distance sets on $S^{d-1}$. 
This upper bound is useful for $k$-distance sets to which the linear programming bound is not applicable.

In Section \ref{section locally}, we discuss locally two-distance sets in $\mathbb{R}^d$. 
We first give an upper bound for the cardinalities of locally two-distance sets. 
Moreover, we mention that every proper locally two-distance set in $\mathbb{R}^d$ with more than $d(d+1)/2$ points   
contains a two-distance set in $S^{d-2}$ which attains the Fisher type upper bound. 
Note that a two-distance set in $\mathbb{R}^d$ with $d(d+1)/2$ points exists. 
We also classify optimal locally two-distance sets in $\mathbb{R}^d$ for $d<8$.  
In addition, we determine $LDS^\ast _2(d)$ for $d<40$ by using the value of $DS^{\ast}_d(2)$ for $d<40$. 
In particular, we  do not know $DS^{\ast}_{23}(2)$ but can determine $LDS^{\ast}_{23}(2)$.

\section{Locally distance sets and weighted spherical designs } \label{sphere}
We prove Theorem \ref{weighted-locally} in this section.  First, we give the definition of weighted spherical designs. 
\begin{defn}[Weighted spherical designs]
Let $X$ be a finite set on $S^{d-1}$. Let $w$ be a weight function: $w: X \rightarrow \mathbb{R}_{>0}$, such that $\sum_{x\in X} w(x)=1$. 
$(X,w)$ is called a weighted spherical $t$-design if the following equality holds for any polynomial $f$ in $d$ variables and of degree at most $t$:
$$
\frac{1}{|S^{d-1}|}\int_{S^{d-1}} f(x) d\sigma(x) = \sum_{x \in X}w(x)f(x),
$$
where the left hand side involves the integral of $f$ on the sphere. $X$ is called a spherical $t$-design if $w(x)=1/|X|$ for all $x\in X$. 
\end{defn} 
We have the following lower bound for the cardinalities of weighted spherical $t$-designs. 
\begin{thm}[Fisher type inequality \cite{Del, Delsarte-Seidel} ]
$(\mathrm{i})$ Let $X$ be a weighted spherical $2e$-design. Then, $|X| \geq \binom{d+e-1}{e}+\binom{d+e-2}{e-1}=N_d(e)$.\\
$(\mathrm{ii})$ Let $X$ be a weighted spherical $(2e-1)$-design. Then, $|X| \geq 2 \binom{d+e-2}{e-1}=N^{\prime}_d(e)$.  
\end{thm}
If equality holds, $X$ is said to be tight. The following theorem shows a strong relationship between tight spherical $t$-designs and 
$k$-distance sets.
\begin{thm}[Delsarte, Goethals and Seidel \cite{Del}]
$(\mathrm{i})$ $X$ is a $k$-distance set on $S^{d-1}$ with $N_d(k)$ points if and only if $X$ is a tight spherical $2k$-design.  \\
$(\mathrm{ii})$ $X$ is an antipodal $k$-distance set on $S^{d-1}$ with $N_d^{\prime}(k)$ points if and only if $X$ is a tight spherical $(2k-1)$-design.  
\end{thm}  
\begin{rem} \label{tight rem}
In particular, $X$ is a two-distance set on $S^{d-1}$ with $N_d(2)$ points if and only if $X$ is a tight spherical four-design. 
$X$ is an antipodal three-distance set on $S^{d-1}$ with $N_d^{\prime}(2)$ points if and only if $X$ is a tight spherical five-design. 
Note that the existence of a tight spherical four-design on $S^{d-2}$ is equivalent to the existence of a tight spherical five-design on $S^{d-1}$. 
Let $X$ be a tight spherical five-design on $S^{d-1}$. Then, we can put $A(X)=\{ \alpha, \beta, 2 \}\, (\alpha < \beta)$. 
For a fixed $x \in X$, we define $X_{\alpha}:=\{y \in X\mid d(x,y)=\alpha \}$. 
Then, we can regard $X_\alpha$ as a tight spherical four-design on $S^{d-2}$. 
This relationship between tight four-designs and five-designs is important in Section \ref{section locally}.    
\end{rem}
 Let $X=\{x_1,x_2,\ldots, x_n\}$ be a finite set on $S^{d-1}$. Let ${\rm Harm}_l(\mathbb{R}^d)$ be the linear space of all real harmonic homogeneous polynomials of degree $l$, in $d$ variables. We put $h_l:=\dim({\rm Harm}_l(\mathbb{R}^d))$. Let $\{\varphi_{l,i}\}_{i=0,1,\ldots, h_l}$ be an orthonormal basis of ${\rm Harm}_l(\mathbb{R}^{d-1})$ with respect to the inner product $\langle f, g \rangle= \frac{1}{|S^{d-1}|}\int_{S^{d-1}} f(x)g(x)d\sigma(x)$.
Let $H_l$ be the characteristic matrix of degree $l$, that is, its $(i,j)$-th entry is $\varphi_{l,j}(x_i)$. The following gives the definition of Gegenbauer polynomials and discusses the Addition
Formula which will be used in the succeeding discussion. 
 \begin{defn}
Gegenbauer polynomials are a set of
orthogonal polynomials $\{G_l^{(d)}(t) \mid l=1,2,\ldots \}$ of one variable $t$.
For each $l$, $G_l^{(d)}(t)$ is a 
polynomial of degree $l$, defined in the following manner.
\begin{enumerate}
\item $G_0^{(d)}(t) \equiv 1$, $G_1^{(d)}(t)=d t$.
\item $t G_l^{(d)}(t)=\lambda_{l+1}G_{l+1}^{(d)}(t)+(1-\lambda_{l-1})G_{l-1}^{(d)}(t)$ for $l \geq 1$, where $\lambda_l=\frac{l}{d+2l-2}$. 
\end{enumerate}
\end{defn}
Note that $G_l^{(d)}(1)=\dim({\rm Harm}_l(\mathbb{R}^d))=h_l$. Let $(,)$ be the standard inner product in $\mathbb{R}^d$.
\begin{thm}[Addition formula \cite{Del,Ban}] 
For any $x$, $y$ on $S^{d-1}$, we have
$$
\sum_{k=1}^{h_l} \varphi_{l,k}(x) \varphi_{l,k}(y)= G_l^{(d)}((x,y)).
$$
\end{thm} 
Let $I$ be the identity matrix, and ${}^t N$ be the transpose of a matrix $N$. 
 The following is a key theorem to prove Theorem \ref{main of sphere}.  
\begin{thm} \label{lemma of H}
The following are equivalent: \\
$(\mathrm{i})$ $(X,w)$ is a weighted spherical $t$-design.\\
$(\mathrm{ii})$ $^tH_{e} W H_{e}=I$ and $^tH_{e} W H_{r}=0$ for $ e = \lfloor \frac{t}{2} \rfloor$ and $r=e-(-1)^t$. Here, $W={\rm Diag}\{ w(x_1),w(x_2),\ldots,w(x_{n})\}$.  
\end{thm}
We require the two following lemmas in order to prove Theorem \ref{lemma of H}. 
\begin{lem}[Lemma 3.2.8 in \cite{Ban} or \cite{Del}] \label{coefficient}
We have the Gegenbauer expansion 
$G_k^{(d)} G_{l}^{(d)} =\sum_{i=0}^{k+l} q_i(k,l) G_{i}^{(d)}$. 
Then, the following hold.\\ 
$(\mathrm{i})$ For any $i,k$ and $l$, $q_i(k,l) \geq 0$. \\
$(\mathrm{ii})$ For any $k$ and $l$, $q_0(k,l) = h_k \delta_{k,l}$, where $\delta_{k,l}=1$ if $k=l$ and $\delta_{k,l}=0$ if $k \ne l$.  \\
$(\mathrm{iii})$ $q_i(k,l) \ne 0$ if and only if $|k-l| \leq i \leq k+l$ and $i \equiv k+l \mod 2$. 
\end{lem}
For an $m \times n$ matrix $M$, we define $||M||^2 := \sum_{i=1}^m \sum_{j=1}^n M(i,j)^2$, namely the sum of squares of all matrix entries.  
\begin{lem} \label{2.9}
For $k+l \geq 1$, 
\begin{eqnarray}
\left| \left|^tH_{k} W H_{l} -  \Delta_{k,l} \right| \right|^2= \sum_{i=1}^{k+l}q_i(k,l) \left| \left|^tH_{i} W H_{0} \right| \right|^2 
\end{eqnarray}
where 
$$
\Delta_{k,l}= \begin{cases}
I, \text{ if } k=l\\
0, \text{ if } k \ne l
\end{cases}.
$$
\end{lem}
\begin{proof}
Note that 
\begin{eqnarray}
& &||^tH_{k} W H_{l}||^2 =  \sum_{i=1}^{h_{k}} \sum_{j=1}^{h_{l}}\biggl(\sum_{x \in X} w(x) \varphi_{k,i}(x)  \varphi_{l,j}(x) \biggr)^2 \\
&=& \sum_{x \in X} \sum_{y \in X} w(x)w(y) \sum_{i=1}^{h_{k}} \varphi_{k,i}(x) \varphi_{k,i}(y) \sum_{j=1}^{h_{l}} \varphi_{l,j}(x) \varphi_{l,j}(y) \\
&=& \sum_{x\in X} \sum_{y \in X} w(x)w(y) G_{k}^{(d)}( ( x , y ) )  G_{l}^{(d)}( ( x , y ) ). \nonumber
\end{eqnarray} 
When $l=0$, we have 
\begin{eqnarray} \label{3.37}
||^tH_{k} W H_{0}||^2=\sum_{x\in X} \sum_{y \in X} w(x)w(y)  G_{k}^{(d)}( ( x , y ) ).  
\end{eqnarray}
If $k \ne l $, then
\begin{eqnarray}
||^tH_{k} W H_{l}||^2 &=& \sum_{x\in X} \sum_{y \in X} w(x)w(y)  G_{k}^{(d)}( ( x , y))  G_{l}^{(d)}( ( x , y ) ) \nonumber \\
&=& \sum_{x\in X} \sum_{y \in X} w(x)w(y)  \sum_{i=0}^{k+l}q_i(k,l)G_{i}^{(d)}( ( x , y ))  \nonumber \\
&=& \sum_{i=0}^{k+l}q_i(k,l) ||^tH_{i} W H_{0}||^2 \qquad (\because \text{equality } (\ref{3.37})) \nonumber \\
&=&\sum_{i=1}^{k+l}q_i(k,l) ||^tH_{i} W H_{0}||^2 \qquad (\because \text{Lemma }\ref{coefficient}). \nonumber 
\end{eqnarray}
If $k = l $, then the summation of the squares of the diagonal entries is 
\begin{eqnarray}
& &\sum_{i=1}^{h_{k}}\biggl( (  {}^tH_{k} W H_{k} -  I )(i,i) \biggr)^2 = \sum_{i=1}^{h_{k}} \biggl( \sum_{x\in X} w(x)  \varphi_{k,i}(x) \varphi_{k,i}(x) - 1  \biggr)^2 \nonumber \\
& = & \sum_{i=1}^{h_{k}}\Biggl( \biggl( \sum_{x\in X} w(x)  \varphi_{k,i}(x) \varphi_{k,i}(x) \biggr)^2   -2  \sum_{x\in X} w(x) \varphi_{k,i}(x) \varphi_{k,i}(x) +1 \Biggr) \nonumber \\
& = & \sum_{i=1}^{h_{k}} \biggl( \sum_{x\in X} w(x) \varphi_{k,i}(x) \varphi_{k,i}(x) \biggr)^2   -  2 \sum_{x\in X} w(x)  \sum_{i=1}^{h_{k}} \varphi_{k,i}(x) \varphi_{k,i}(x )  + h_{k}  \nonumber \\
& = & \sum_{i=1}^{h_{k}} \biggl( \sum_{x\in X} w(x)  \varphi_{k,i}(x) \varphi_{k,i}(x) \biggr)^2   -  2 \sum_{x\in X} w(x)    G_{k}^{(d)}(1) + h_{k}    \nonumber \\
&= & \sum_{i=1}^{h_{k}} \biggl( \sum_{x\in X} w(x)  \varphi_{k,i}(x) \varphi_{k,i}(x) \biggr)^2 - h_{k} \nonumber
\end{eqnarray}
Therefore, 
\begin{eqnarray}
 \left|\left|^tH_{k} W H_{k} - I \right|\right|^2 &=& ||^tH_{k} W H_{k}||^2  - h_{k} \nonumber \\
&=& \sum_{i=0}^{2 k}q_i(k,k) ||^tH_{i} W H_{0}||^2- h_{k} \nonumber \\
&=& \sum_{i=1}^{2 k}q_i(k,k) ||^tH_{i} W H_{0}||^2. 
\end{eqnarray}
\end{proof}

\begin{proof}[Proof of Theorem \ref{lemma of H}]
$\mathrm{(i)} \Rightarrow \mathrm{(ii)}$ is clear. We prove $\mathrm{(ii)} \Rightarrow \mathrm{(i)}$. 
 By Lemma \ref{2.9},
\begin{eqnarray} 
 \left|\left|^tH_{e} W H_{e} -  I \right|\right|^2= \sum_{i=1}^{2e} q_i(e,e) \left|\left|^tH_{i} W H_{0}\right|\right|^2 =0.
\end{eqnarray} 
 We have $^tH_{i} W H_{0} =0$ for even $i \leq t$, because $q_i(e,e)>0$ for even $i$, and $q_i(e,e)=0$ for odd $i$. On the other hand,  
\begin{eqnarray} 
 \left|\left|^tH_{e} W H_{r}\right|\right|^2= \sum_{i=1}^{2e-(-1)^t}q_i(e,r) \left|\left|^tH_{i} W H_{0}\right|\right|^2 =0.
\end{eqnarray} 
We have $^tH_{i} W H_{0} =0$ for odd $i \leq t$, because $q_i(e,r)>0$ for odd $i$, and $q_i(e,r)=0$ for even $i$.
 Therefore, these imply that for any $f \in {\rm P}_t(S^{d-1})$, the following equality holds: 
$$
 \frac{1}{|S^{d-1}|} \int_{S^{d-1}} f(x) d \sigma(x)= \sum_{x \in X} w(x)f(x). 
$$  
\end{proof}

\begin{proof}[Proof of Theorem \ref{weighted-locally}]
Let $X=\{x_1,x_2,\ldots , x_n\}$ be a locally $k$-distance set on $S^{d-1}$. Suppose $|X|=N_d(k)$.  For each $x \in X$, we define $A_{\textup{inn}}(x):=\{(x,y) \mid  y\in X, x \ne y \}$. 
For each $x \in X$, we define the polynomial in $d$ variables:  
$$
F_x(\xi):= (x, \xi )^{k-|A_{\textup{inn}}(x)|}\prod_{\alpha \in A_{\textup{inn}}(x)} \frac{(x,\xi)-\alpha}{1-\alpha},
$$
where $\xi=(\xi_1,\xi_2,\ldots,\xi_d)$.
$F_x(\xi)$ is of degree $k$ for all $x \in X$. For all $x_i,x_j \in X$, $F_{x_i}(x_j)=\delta_{i,j}$. 
We have the Gegenbauer expansion: 
$$
F_x(\xi)=\sum_{i=0}^k f_i^{(x)} G_i^{(d)}((x,\xi))
$$ 
where $f_i^{(x)}$ are real numbers. In particular, we remark that $f_k^{(x)}>0$ for every $x \in X$. By the addition formula, 
\begin{eqnarray} \label{eq}
F_x(\xi)=\sum_{i=0}^k f_i^{(x)} G_i^{(d)}((x,\xi))=\sum_{i=0}^k f_i^{(x)} \sum_{j=1}^{h_i} \varphi_{i,j}(x)\varphi_{i,j}(\xi)
\end{eqnarray}
 for $\xi \in S^{d-1}$. 
 We define the diagonal matrices $C_i:= {\rm Diag} \{f_{i}^{(x_1)}, f_i^{(x_2)}, \ldots , f_i^{(x_n)} \}$ for $0 \leq i \leq k$.  
$[C_0 H_0 , C_1 H_1, \ldots ,C_k H_k ]$ and $[H_0, H_1, \ldots H_k]$ are $n \times n$ matrices. 
By the equality (\ref{eq}), we have the equality: 
\begin{eqnarray}
[C_0 H_0 , C_1 H_1, \ldots ,C_k H_k ] \left[\begin{array}{c} 
^t{H_0}\\
^t{H_1}\\
\vdots \\
^t {H_k}
\end{array}
\right] = [F_{x_i}(x_j)]_{i,j}=I.
\end{eqnarray} 
Therefore, $[C_0 H_0 , C_1 H_1, \ldots ,C_k H_k ]$ and $[H_0, H_1, \ldots H_k]$ are non-singular matrices. Thus, 
\begin{eqnarray}
\left[\begin{array}{c} 
^t{H_0}\\
^t{H_1}\\
\vdots \\
^t {H_k}
\end{array}\right][C_0 H_0 , C_1 H_1, \ldots ,C_k H_k ] 
 &=& I \\
 \left[ 
\begin{array}{cccc}
^t H_0C_0H_0 & ^t H_0C_1H_1 & \cdots & ^tH_0C_kH_k \\  
^t H_1C_0H_0 & ^t H_1C_1H_1 & \cdots & ^tH_1C_kH_k \\
\vdots &\vdots&\ddots&\vdots \\
^t H_kC_0H_0 & ^t H_kC_1H_1 & \cdots & ^tH_kC_kH_k
 \end{array}
\right]&=& I.
\end{eqnarray} 
Therefore, $^tH_kC_kH_k=I$ and $^tH_{k-1}C_kH_k=0$. If we define the weight function $w(x):=f_k^{(x)}$ for $x\in X$, then $X$ is a tight weighted spherical $2k$-design on $S^{d-1}$ by Theorem \ref{lemma of H}.  \\

\noindent 
{\it Antipodal case} \quad 
Let $X$ be an antipodal $k$-distance set with $N_d^{\prime}(k)$ on $S^{d-1}$. There exist a subset $Y$ such that $X= Y \cup (-Y)$ and $|X|=2|Y|$. We define $A_{\textup{inn}}^2(x):=\{(x,y)^2 \mid y \in X, y \ne \pm x \}$ and 
$$
\varepsilon=\begin{cases}
1, \text{ if $k$ is even},  \\
0, \text { if $k$ is odd}.
\end{cases} 
$$
For each $y \in Y$, we define the  polynomial in $d$ variables 
$$
F_y(\xi):= (y, \xi)^{k-1-2|A_{\textup{inn}}^2(y)\setminus \{0\}|} \prod_{0 \ne \alpha^2 \in A_{\textup{inn}}^2(y)} \frac{(y,\xi)^2-\alpha^2}{1-\alpha^2}. 
$$
$F_y(\xi)$ is of degree $k-1$ for all $y \in Y$. For all $y_i,y_j \in Y$, $F_{y_i}(y_j)=\delta_{i,j}$.  
We have the Gegenbauer expansion: 
$$
F_y(\xi)=\sum_{i=0}^{k-1} f_i^{(y)} G_i^{(d)}((y,\xi)).
$$ 
Note that $f_i = 0$ for $i \equiv k \mod 2$.
 In particular, we remark that $f_{k-1}^{(y)}>0$ for every $y \in Y$. We define the diagonal matrices $C_i:= {\rm Diag} \{f_{i}^{(y_1)}, f_i^{(y_2)}, \ldots , f_i^{(y_{n/2})} \}$ for $0 \leq i \leq k-1$. Let $H^{(Y)}_l$ be the characteristic matrix with respect to $Y$. 
$[C_\varepsilon H^{(Y)}_\varepsilon , C_{\varepsilon+2} H^{(Y)}_{\varepsilon +2}, \ldots ,C_{k-1} H^{(Y)}_{k-1} ]$ and $[H^{(Y)}_\varepsilon, H^{(Y)}_{\varepsilon+2}, \ldots ,H^{(Y)}_{k-1}]$ are ${n/2} \times {n/2}$ matrices. 
By the addition formula, we have the equality: 
\begin{eqnarray}
[C_\varepsilon H^{(Y)}_\varepsilon , C_{\varepsilon+2} H^{(Y)}_{\varepsilon +2}, \ldots ,C_{k-1} H^{(Y)}_{k-1} ] \left[\begin{array}{c} 
^t{H^{(Y)}_\varepsilon}\\
^t{H^{(Y)}_{\varepsilon+2}}\\
\vdots \\
^t {H^{(Y)}_{k-1}}
\end{array}
\right] = I.
\end{eqnarray} 
Therefore, $[C_\varepsilon H^{(Y)}_\varepsilon , C_{\varepsilon+2} H^{(Y)}_{\varepsilon +2}, \ldots ,C_{k-1} H^{(Y)}_{k-1} ] $ and $[H^{(Y)}_\varepsilon, H^{(Y)}_{\varepsilon+2}, \ldots ,H^{(Y)}_{k-1}]$ are non-singular matrices. Thus, 
\begin{eqnarray}
\left[\begin{array}{c} 
^t{H^{(Y)}_\varepsilon}\\
^t{H^{(Y)}_{\varepsilon+2}}\\
\vdots \\
^t {H^{(Y)}_{k-1}}
\end{array}\right][C_\varepsilon H^{(Y)}_\varepsilon , C_{\varepsilon+2} H^{(Y)}_{\varepsilon +2}, \ldots ,C_{k-1} H^{(Y)}_{k-1} ] 
 &=& I \\
 \left[ 
\begin{array}{cccc}
^t H^{(Y)}_\varepsilon C_\varepsilon H^{(Y)}_\varepsilon & ^t H^{(Y)}_\varepsilon C_{\varepsilon+2} H^{(Y)}_{\varepsilon+2} & \cdots & ^t H^{(Y)}_\varepsilon C_{k-1} H^{(Y)}_{k-1} \\  
^t H^{(Y)}_{\varepsilon+2} C_\varepsilon H^{(Y)}_\varepsilon & ^t H^{(Y)}_{\varepsilon +2} C_{\varepsilon+2} H^{(Y)}_{\varepsilon+2} & \cdots & ^t H^{(Y)}_{\varepsilon +2} C_{k-1} H^{(Y)}_{k-1} \\
\vdots &\vdots&\ddots&\vdots \\
^t H^{(Y)}_{k-1}C_\varepsilon H^{(Y)}_\varepsilon & ^t H^{(Y)}_{k-1}C_{\varepsilon+2}H^{(Y)}_{\varepsilon+2} & \cdots & ^tH^{(Y)}_{k-1}C_{k-1}H^{(Y)}_{k-1}
 \end{array}
\right]&=& I.
\end{eqnarray} 
Therefore, $^tH^{(Y)}_{k-1}C_{k-1}H^{(Y)}_{k-1}=I$. Let $H_l$ be a characteristic matrix with respect to $X$. We select the weight function $w(x):=f_{k-1}^{(x)}/2$ and $w(-x)=w(x)$ for $x \in X$. Since $X$ is antipodal, this implies $^tH_{k-1}W H_{k-1}=I$ and $^tH_{k-1}W H_{k}= 0$. Therefore, $X$ is a tight weighted spherical $(2k-1)$-design by Theorem \ref{lemma of H}. 

$(\Leftarrow)$ It is known that tight weighted spherical $2k$-designs (resp.\ $(2k-1)$-design) are tight spherical $2k$-design (resp.\ $(2k-1)$-design) \cite{Taylor, Bannai-Bannai, Etsuko2}. 
Therefore, a tight weighted spherical $2k$-design (resp.\ $(2k-1)$-design) is a $k$-distance set (resp.\ antipodal $k$-distance set). 
\end{proof} 

Theorem \ref{weighted-locally} implies that (resp.\ antipodal) locally $k$-distance sets attaining the Fisher type upper bound are (resp.\ antipodal) $k$-distance sets .

\section{A new upper bound for $k$-distance sets on $S^{d-1}$} \label{section bound}
The following upper bound for the cardinalities of $k$-distance sets is well known. 
\begin{thm}[Linear programming bound \cite{Del}]
Let $X$ be a $k$-distance set on $S^{d-1}$. We define the polynomial $F_X(t):=\prod_{\alpha \in A_{\textup{inn}}(X)}(t-\alpha)$ for $X$ where $A_{\textup{inn}}(X):=\{(x,y) \mid x, y \in X , x\ne y  \}$. 
We have the Gegenbauer expansion 
$$
F_X(t)=\prod_{\alpha \in A_{\textup{inn}}(X)}(t-\alpha)= \sum_{i=0}^k f_i G_i^{(d)}(t),
$$
where $f_i$ are real numbers. If $f_0>0$ and $f_i\geq 0$ for all $1 \leq i \leq k$, then
$$
|X| \leq \frac{F_X(1)}{f_0}.
$$
\end{thm}
This upper bound is very useful when $A_{\textup{inn}}(X)$ is given. However, if some $f_i$ happens to be negative, then we have no useful upper bound for the cardinalities of $k$-distance sets. In this section, we give a useful upper bound for this case. 
A proof of the following theorem builds upon Delsarte's ideas for the binary codes \cite{Del2}. 
   \begin{thm} \label{main}
Let $X$ be a $k$-distance set on $S^{d-1}$.  
We define the polynomial $F_X(t)$ of degree $k$:  
$$F_X(t):=\prod_{\alpha \in A_{\textup{inn}}(X)} (t-\alpha)= \sum_{i=0}^k f_i G_i^{(d)}(t),$$
where $f_i$ are real numbers. Then, 
\begin{eqnarray}
|X| \leq \sum_{i \text{ with } f_i>0} h_i, 
\end{eqnarray} 
where the summation is over $i$ with $0\leq i \leq k$ satisfying $f_i>0$ and $h_i=\dim({\rm Harm}_i(\mathbb{R}^d))$.  
\end{thm}

\begin{proof}
Let $X:=\{x_1,x_2,\ldots,x_{n}\}$ be a $k$-distance set on $S^{d-1}$. 
Let $\{ \varphi_{l,k}\}_{1\leq k \leq h_l}$ be an orthonormal basis of ${\rm Harm}_l(\mathbb{R}^d)$.
$H_l$ is the characteristic matrix. We have the Gegenbauer expansion $F_X(t)=\prod_{\alpha \in A_{\textup{inn}}(X)}\frac{t-\alpha}{1-\alpha}= \sum_{i=0}^k f_i G_i^{(d)}(t)$. Define the $ \sum_{i=0}^k h_i \times n $ matrix $H:={}^t [ H_0,  H_1, \ldots, H_k]$. 
By the addition formula, we get 
\[
	  {}^t H F H=I_{n}
\]
where $I_m$ is the identity matrix of degree $m$, and $F=f_0 I_{1} \oplus f_1 I_{h_1} \oplus \cdots \oplus f_s I_{h_s}$ (direct sum). Therefore, the column vectors of $H$ are linearly independent, and lie in the positive subspace of the quadratic form $F$. Thus, $n$ can not exceed the number of the positive entries of $F$.
\end{proof}
If $f_i>0$ for all $0\leq i\leq k$, then this upper bound is the same as the Fisher type inequality. 

By using a similar method, we prove a similar upper bound for the antipodal case. 
\begin{thm}[Antipodal case] \label{main anti}
 Let $X$ be an antipodal $k$-distance set on $S^{d-1}$.  
 We define the polynomial $F_X(t)$ of degree $k-1$:  
$$F_X(t):=\prod_{\alpha \in A_{\textup{inn}}(X) \setminus \{-1 \}} (t-\alpha)= \sum_{i=0}^{k-1} f_i G_i^{(d)}(t),$$
where the $f_i$ are real and $f_i=0$ for $i\equiv k \mod 2$. Then, 
\begin{eqnarray}
|X| \leq 2 \sum_{i \text{ with } f_i>0} h_i. 
\end{eqnarray}  
\end{thm}

\begin{cor} \label{musin}
Let $X$ be a two-distance set and $A_{\textup{inn}}(X)=\{ \alpha, \beta \}$. Then, 
$F_X(t):=(t-\alpha )(t-\beta)=\sum_{i=0}^2 f_i G_i^{(d)}(t)$ where $f_0=\alpha \beta + 1/d$, $f_1=-(\alpha+\beta)/d$ and $f_2=2/(d(d+2))$. 
If $\alpha +\beta \geq 0$, then 
$$
|X| \leq h_0 + h_2= \binom{d+1}{2}.
$$  
\end{cor} 
Musin proved this corollary by using a polynomial method in \cite{Musin}. This corollary is used in proof of Theorem \ref{partial ans} in this paper. The following examples attain this upper bound in Corollary \ref{musin}. 
\begin{exmp} \label{mid regular}
Let $U_d$ be a $d$-dimensional regular simplex. We define
$$X:=\left\{ \left. \frac{x+y}{2} \right| x,y \in U_d, x \ne y \right\}$$  for $d \geq 7$.
Then, $X$ is a two-distance set on $S^{d-1}$, $|X|=d(d+1)/2$, $f_0>0$, $f_1\leq 0$ and $f_2>0$.  
\end{exmp}

Let us introduce some examples which attain the upper bounds in Theorem \ref{main} and \ref{main anti}.

\begin{cor} 
Let $X$ be a one-distance set and $A_{\textup{inn}}(X)=\{ \alpha \}$. Then, 
$F_X(t):=t-\alpha =\sum_{i=0}^1 f_i G_i^{(d)}(t)$ where $f_1= 1/d$ and $f_0=-\alpha$. 
If $\alpha \geq 0$, then 
$$
|X| \leq  h_1= d.
$$  
\end{cor} 
Clearly, a $d$-point $(d-1)$-dimensional regular simplex with a nonnegative inner product on $S^{d-1}$ attains this upper bound.  

\begin{cor} 
Let $X$ be an $k$-distance set on $S^{d-1}$. We have the Gegenbauer expansion $F_X(t)=\prod_{\alpha \in A_{\textup{inn}}(X)}(t-\alpha)=\sum_{i=0}^k f_i G_i^{(d)}(t)$. If $f_i > 0$ for all $i \equiv k \mod 2$ and $f_i \leq 0$ for all $i \equiv k-1 \mod 2$, then 
$$
|X| \leq \sum_{i=0}^{\lfloor \frac{k}{2} \rfloor}h_{k-2i}= \binom{d+k-1}{k}.
$$
\end{cor} 
The following examples attain their upper bounds. 
\begin{exmp}
Let $X$ be a tight spherical $(2k-1)$-design, that is, $X$ is an antipodal $k$-distance set with $N_d^{\prime}(k)$ points. There exist a subset $Y$ such that $X= Y \cup (-Y)$ and $|X|=2|Y|$. $Y$ is an $(k-1)$-distance set with $\binom{d+k-2}{k-1}$ points. Defining $F_Y(t):=\sum_{i=0}^{k-1}f_iG_i^{(d)}(t)$, we have $f_i = 0$ for all $i \equiv k \mod 2$, and $f_i > 0$ for all $i \equiv k-1 \mod 2$. 
\end{exmp}


\section{Locally two-distance sets} \label{section locally}
In this section, we will consider locally two-distance sets. 
Recall that a locally two-distance set is said to be {\it proper} if it is not a two-distance set. 
The following examples imply that there are infinitely many proper locally two-distance sets when their cardinalities are small for their dimensions.  

\begin{exmp}
Let $U_d$ be the vertex set of  a regular simplex in $\mathbb{R}^d$ and $O$ be the center of the regular simplex. 
Let $y$ be a point on the line passing through $x\in U_d$ and $O$. Then $U_d\cup \{y\}$ is a locally two-distance set. 
Except for finitely many exceptions, such locally two-distance sets are proper. 
\end{exmp}

\begin{exmp}
Let $\{{e_1}, {e_2}, \ldots , {e_d}\}$ be an orthonormal basis of $\mathbb{R}^d$. 
Let $$X=\{x_1, y_1, x_2, y_2, \ldots ,x_{k-1}, y_{k-1}\}$$ where $$x_1={e_1}, \quad y_1=-{e_1}$$ and 
$$jx_j=e_{2j-2}+\sqrt{j^2-1}e_{2j-1}, \quad jy_j=e_{2j-2}-\sqrt{j^2-1}e_{2j-1}$$ for $2\leq j \leq k-1$. 
Then $X$ is a locally two-distance set and a $k$-distance set in $\mathbb{R}^{2k-3}$. 
\end{exmp}

\subsection{An upper bound for the cardinalities of locally two-distance sets}

\noindent 
\begin{lem}\label{lemdr}
$(\mathrm{i})$ Let $X\subset \mathbb{R}^d$ be a locally two-distance set with at least $d+2$ points. 
If $d\geq 2$, then there exist points $x, x' \in X$ ($x\ne x'$) such that $A(x)=A(x')=\{\alpha , \alpha '\}$ for some $\alpha , \alpha '\in \mathbb{R}_{>0}$ 
$(\alpha \ne \alpha ')$. \\
$(\mathrm{ii})$ Let $X$ be a locally two-distance set in $\mathbb{R}^d$ with $n\geq d+2$ points. 
Then there exists $Y\subset X$ with $|Y|=n-d$ and $|A(x)|=2$ for any $x\in Y$. 
\end{lem}

\begin{proof}
(i)  Let $X$ be a locally two-distance set in $\mathbb{R}^d$ with more than $d+1$ points. 
Let $B (\alpha ;x)=\{y\in X | d(x,y)=\alpha \}$ for any $x\in X$ and $\alpha \in A(x)$. 
Since $DS_d(1)=d+1$, there exists $x\in X$ such that $|A(x)|=2$. 
Let $A(x)=\{\alpha _1, \alpha _2\}$, $Y_1=B(\alpha _1; x)$ and $Y_2=B(\alpha _2; x)$. 
For $y_1\in Y_1$ and $y_2\in Y_2$, if $d(y_1,y_2)\in \{\alpha _1, \alpha _2\}$, 
then we have $A(x)=A(y_1)$ or $A(x)=A(y_2)$ and this lemma holds. 
Otherwise, there exists $\beta \notin \{\alpha _1,\alpha _2\}$ such that 
$d(y_1,y_2)=\beta$ for all $y_1\in Y_1$ and $y_2\in Y_2$. 
Thus $A(y_i)=\{\alpha _i, \beta\}$ for any $y_i\in Y_i$ ($i=1, 2$). 
Moreover, $|Y_1|\geq 2$ or $|Y_2|\geq 2$ since $|X|\geq 4$. \\
(ii)  Let $X$ be a locally two-distance set in $\mathbb{R}^d$ with $n\geq d+2$ points. 
Let $Y'$ be the set of all points in $X$ with $|A(x)|=1$. 
Then clearly $A(x)=A(x')$ for any $x,x'\in Y'$. 
Therefore $Y'$ is a one-distance set and $|Y'|\leq d+1$. 
Moreover if $|Y'|=d+1$, then $Y'\cup \{y\}$ must be a 
one-distance set for any $y\in X\setminus Y'$, which is a contradiction.  
Thus $|Y'|\leq d$ and $|X\setminus Y'|\geq n-d$. 
\end{proof}

\begin{rem}\label{lower}
When we consider optimal locally two-distance sets, the condition $|X|\geq d+2$ in Lemma \ref{lemdr} is not so important 
because there is a lower bound $d(d+1)/2\leq DS_d(2) \leq LDS_d(2)$ (cf.\ Example \ref{mid regular}). 
\end{rem}

Let $X$ be a locally two-distance set. 
A subset $Y\subset X$ is called a {\it saturated subset} if $|Y|\geq 2$ and $Y$ is a maximal subset such that there exists $\alpha $, $\beta $ ($\alpha \ne \beta $) 
with $A_X(y)=\{\alpha , \beta \}$  for any $y\in Y$. 
Lemma \ref{lemdr} assures us that every locally two-distance set in $\mathbb{R}^d$ with at least $d+2$ points 
contains a saturated subset. Let $Y=\{y_1, y_2, \ldots y_m\}\subset X$ be a saturated subset.  
Then $Y$ is a two-distance set and $X\setminus Y$ is a locally two-distance set in the space $\{x\in \mathbb{R}^d|d(y_1,x)=d(y_2,x)=\cdots =d(y_m,x)\}$ by maximality. 
If $X\setminus Y\ne \emptyset$, then all points in $Y$ are on a common sphere. 
Moreover $Y\cup \{x\}$ is a two-distance set for any $x\in X\setminus Y$. 

\begin{lem}\label{dimension}
Let $Y=\{y_0, y_1, \ldots, y_{m-1}\}\subset \mathbb{R}^d$. 
Without loss of generality, we may assume that $y_0$ is the origin of $\mathbb{R}^d$. 
Let $\dim(Y)$ be the dimension of the space spanned by $Y$ 
and $\Sol(Y)=\{x\in \mathbb{R}^d|d(y_0,x)=d(y_1,x)=\cdots =d(y_{m-1},x)\}$. 
Then $\Sol(Y)$ is contained in a $(d-\dim(Y))$-dimensional affine subspace if $\Sol(Y)\ne \emptyset$.
\end{lem}

\begin{proof}
Let $y_i=(y_{i1},y_{i2},\ldots ,y_{id})$ for $1\leq i\leq m-1$ and let $x=(x_1,x_2,\ldots ,x_d)$. 
For $1\leq i \leq m-1$, $d(y_i,x)=d(y_0,x)$ implies 
$$\sum _{k=1}^d y_{i\, k}x_k=\frac{1}{2}\sum _{k=1}^d {y_{i\, k}}^2.$$
Therefore
$$\Sol(Y)=\left\{x\in \mathbb{R}^d|
\begin{pmatrix}
y_{1\, 1}&y_{1\, 2}&\cdots &y_{1\, d}\\
y_{2\, 1}&y_{2\, 2}&\cdots &y_{2\, d}\\
\vdots &\vdots &\ddots &\vdots \\
y_{m-1\, 1}&y_{m-1\, 2}&\cdots &y_{m-1\, d}
\end{pmatrix}
\begin{pmatrix}
x_1\\x_2\\ \vdots \\ x_d
\end{pmatrix}
=
\begin{pmatrix}
c_1\\c_2\\ \vdots \\ c_d
\end{pmatrix}\right\}
$$
where 
$$c_i=\frac{1}{2}\sum _{k=1}^d{y_{i\, k}}^2.$$
Since the rank of the above matrix is $\dim(Y)$, $\Sol(Y)$ is contained in a $(d-\dim(Y))$-dimensional subspace 
if $\Sol(Y)\ne \emptyset$. 
\end{proof}

By Lemma \ref{dimension}, the following lemma holds. 

\begin{lem}\label{complement}
Let $X$ be a locally two-distance set in $\mathbb{R}^d$. 
Let $Y\subset X$ be a saturated subset and $\dim(Y)=i$. 
Then $X\setminus Y$ is a locally two-distance set with $\dim(X\setminus Y)\leq d-i$. 
\end{lem}

\begin{rem}\label{dim for proper}
Let $X$ be a locally two-distance set and $Y$ be a saturated subset of $X$ in $\mathbb{R}^d$. 
Then we have $\dim (Y)\ne 0$ by Lemma \ref{lemdr}. 
Moreover, if $\dim (Y)=d$, then $\dim (X\setminus Y)=0$ by Lemma \ref{complement}. 
In this case, $|X\setminus Y|\leq 1$ and $X$ is a two-distance set. 
Therefore  $1\leq \dim (Y)\leq d-1$ for every saturated subset $Y$ of a proper locally two-distance set $X$ in $\mathbb{R}^d$. 
Moreover all points in $Y$ are on a common sphere since $X\setminus Y\ne \emptyset$. 
\end{rem}

From the above remark, we have an upper bound for the cardinality of a proper locally two-distance set. 

\begin{thm}\label{max} Let $X$ be a proper locally two-distance set in $\mathbb{R}^d$. Then 
$$|X|\leq f(d)$$
where
$$f(d)=\max _{1\leq i \leq d-1}\{DS^{\ast}_i(2)+LDS_{d-i}(2) \}.$$
In particular, 
$$LDS_d(2)\leq \max \{DS_d(2), f(d)\}$$
\end{thm}

\begin{proof}
Let $X$ be a proper locally two-distance set in $\mathbb{R}^d$ and 
$Y$ be a saturated subset of $X$ and $i=\dim(Y)$. 
Then $1\leq i\leq d-1$ and all points in $Y$ are on a common sphere by Remark \ref{dim for proper}, so $|Y|\leq DS^{\ast}_i(2)$. 
On the other hand, $|X\setminus Y|\leq LDS_{d-i}(2)$ by Lemma \ref{complement}.
Therefore $|X|\leq DS^{\ast}_i(2)+LDS_{d-i}(2)\leq f(d)$.  
\end{proof}

\begin{cor}\label{width}
Every locally two-distance set in $\mathbb{R}^d$ with at least $d(d+1)/2+3$ points 
is a two-distance set. 
In particular $LDS_d(2)\leq \binom{d+2}{2}$. 
\end{cor}

\begin{proof}
Let $X$ be a proper locally two-distance set in $\mathbb{R}^d$. 
As we will see in Proposition \ref{2dr}, $LDS_d(2)\leq \binom{d+2}{2}$ for small $d$. 
Assume $LDS_i(2)\leq \binom{i+2}{2}$ for any $i\leq d-1$. 
By Theorem \ref{max}, 
\begin{align*}
|X|&\leq \max _{1\leq i\leq d-1}\{DS^{\ast}_i(2)+LDS_{d-i}(2)\} \\
&\leq \max _{1\leq i\leq d-1}\biggl\{\frac{i^2+3i}{2}+\frac{(d-i+2)(d-i+1)}{2} \biggr\}\\
&=\frac{1}{2}\max _{1\leq i\leq d-1}\{2i^2-2di+d^2+3d+2\}\\
&=\frac{d(d+1)}{2}+2\\
\end{align*}
Therefore this corollary holds. 
\end{proof}

\begin{rem}
(i) Since the set of midpoints of a regular simplex in $\mathbb{R}^d$ is a two-distance set with $d(d+1)/2$ points, 
Corollary \ref{width} implies $DS_d(2)\leq LDS_d(2)\leq DS_d(2)+2$. 
For $d\leq 8$, $d\ne 3$, we will see that $DS_d(2)=LDS_d(2)$ in Proposition \ref{2dr}.  \\
(ii) For spherical cases, similarly we have $DS^{\ast}_d(2)\leq LDS^{\ast}_d(2)\leq DS^{\ast}_d(2)+1$. 
\end{rem}

\begin{prob}\label{problem1}
When does $DS_d(2)<LDS_d(2)$ (resp.\ $DS^{\ast}_d(2)<LDS^{\ast}_d(2)$) hold?
\end{prob}

We will give partial results for general cases in Section \ref{sec partial} and give an answer for $d\leq 8$ in Section \ref{sec ans}. 

\subsection{Partial answer to Problem \ref{problem1}}\label{sec partial}
\begin{lem} \label{partial lem1}
$(\mathrm{i})$ Let $X$ be a proper locally two-distance set in $\mathbb{R}^d$ for $d\geq 3$. 
If $d(d+1)/2<|X|$, then there exist $N_{d-1}(2)$-point two-distance set in $S^{d-2}$ 
or $(N_{d-1}(2)-1)$-point two-distance set $Y$ in $S^{d-2}$ with $A(Y)=\{1, 2/\sqrt{3}\}$. \\
$(\mathrm{ii})$ Let $X$ be a proper locally two-distance set in $S^{d-1}$ for $d\geq 3$. 
If $d(d+1)/2<|X|$, then there exist $N_{d-1}(2)$-point two-distance set $Y$ in 
$S^{d-2}$ with $\sqrt{2} \in A(Y)$ or $A(Y)=\{\alpha, \alpha/\sqrt{\alpha^2-1}\}$. 
\end{lem}

\begin{proof}
(i) For the case where $d\in\{3,4\}$, we will prove this proposition directly in Proposition \ref{2dr}. 
Therefore we assume that $d\geq 5$ in this proof. 
Let $X$ be a proper locally two-distance set in $\mathbb{R}^d$ with more than $d(d+1)/2$ points and let $Y$ be a saturated subset of $X$. 
We may assume that $Y$ has maximum cardinality among saturated subsets of $X$.  
Let $i=\dim(Y)$. Then $1\leq i\leq d-1$ since $Y$ is a saturated subset and $X$ is not a two-distance set. 
If $2\leq i\leq d-2$, then $d(d+1)/2\geq |X|$ for $d\geq 5$ by Theorem \ref{max}. 
Moreover if $i=1$, then $|Y|\leq 2$ and $|X\setminus Y|\geq d(d+1)/2-2>d(d-1)+3$ for $d\geq 3$. 
Since $X\setminus Y$ is a locally two-distance set in $\mathbb{R}^{d-1}$, $X\setminus Y$ is a two-distance set by Corollary \ref{width}. 
By Lemma \ref{lemdr}, $X\setminus Y$ contains a saturated subset $Y'$ and $|Y'|> |Y|$. 
This is a contradiction to the assumption. Therefore $i=d-1$. 
Since $|X|\geq d(d+1)/2+1=N_{d-1}(2)+2$ and $|X\setminus Y|\leq LDS_1(2)=3$, $|Y|\geq N_{d-1}(2)-1$.  
It is enough to consider the case $|Y|=N_{d-1}(2)-1$, otherwise $|Y|=N_{d-1}(2)$ and this proposition holds. In this case, $|X\setminus Y|=3$. 
Let $A(Y)=\{\alpha , \beta\}$ and $X\setminus Y=\{x_1, x_2, x_3\}$. 
For any $i\in\{1,2,3\}$, $A(x_i)\ne \{\alpha ,\beta\}$ since $Y$ is a saturated subset.  
Moreover $d(x_i,y)=\alpha $ for all $y\in Y$ or $d(x_i,y)=\beta $ for all $y\in Y$. 
Since $\dim(X\setminus Y)=1$, there are four possibilities for the $x_i$. 
Without loss of generality, we may assume $d(x_1,y)=d(x_2,y)=\alpha$ for all $y\in Y$ 
and $d(x_3,y)=\beta$ for all $y\in Y$. 
Then $d(x_1,x_3)=d(x_2,x_3)=\gamma $ for $\gamma \notin \{\alpha ,\beta\}$ and $d(x_1,x_2)=\alpha $. 
It follows from these conditions that $Y$ is an $(N_{d-1}(2)-1)$-point two-distance set $Y$ 
in $S^{d-2}$ with $A(Y)=\{1, 2/\sqrt{3}\}$. 

\noindent 
(ii) Let $X$ be a proper locally two-distance set in $S^{d-1}$ with more than $d(d+1)/2$ points and let $Y$ be a saturated subset of $X$. 
Similar to the above case, we may assume $i=\dim(Y)=d-1$. 
Since $|X|\geq N_{d-1}(2)+2$ and $|X\setminus Y|\leq LDS^{\ast}_1(2)=2$, $|Y|\geq N_{d-1}(2)$. 
Therefore, $|Y|=N_{d-1}(2)$. 
\end{proof}

\begin{thm} \label{partial ans}
$(\mathrm{i})$ If there exists a proper locally two-distance set $X$ in  $\mathbb{R}^d$ with more than $d(d+1)/2$ points, 
then there exists an $N_{d-1}(2)$-point two-distance set in $S^{d-2}$. \\
$(\mathrm{ii})$ If there exists a proper locally two-distance set $X$ in  $S^{d-1}$ with more than $d(d+1)/2$ points, 
then there exists an $N_{d-1}(2)$-point two-distance set in $S^{d-2}$. 
In particular, a locally two-distance set in $S^{d-1}$ with more than $d(d+1)/2$ points is a subset of a tight spherical five-design. 
\end{thm}
\begin{proof}
(i) Let $X$ be a proper locally two-distance set in $\mathbb{R}^d$ with more than $d(d+1)/2$ points. 
We assume that $X$ does not contain $N_{d-1}(2)$-point two-distance set in $S^{d-2}$. 
Then $X$ contains ($N_{d-1}(2)-1$)-point two-distance set $Y\subset S^{d-2}$ with $A(Y)=\{1, 2/\sqrt{3}\}$ 
by Lemma \ref{partial lem1}(i) . 
However there does not exist such a two-distance set $Y$ by Corollary \ref{musin}. \\
(ii) This is clear by Lemma \ref{partial lem1} (ii) and Remark \ref{tight rem}. 
\end{proof}

\begin{rem} 
Since $ d(d+1)/2 \leq DS_d(2) $ (resp.\ $ d(d+1)/2 \leq DS^\ast _d(2) $), the assumption in Theorem \ref{partial ans} (i) (resp.\  (ii)) 
can be replaced by $ DS_d(2) < LDS_d(2) $ (resp.\ $ DS^\ast _d(2) < LDS^\ast _d(2) $). 
\end{rem}
 
\subsection{Classifications of optimal two-distance sets} 

\noindent 
{\it Euclidean cases}\quad $DS_d(2)$ is determined for $d\leq 8$ and optimal two-distance sets are classified for $d\leq 7$ 
(Kelly \cite{Kelly}, Croft \cite{Croft}, Einhorn-Schoenberg \cite{Ein2} and Lison\v{e}k \cite{Liso}). 
We introduce the results in this subsection. \\

\noindent 
$d=2$: $DS_2(2)$ and the optimal planar two-distance set is isomorphic to the set of vertices of  a regular pentagon 
(Kelly \cite{Kelly}, Einhorn-Schoenberg \cite{Ein2}). 
We denote the set of vertices of the regular pentagon with side length 1 by $R_5$. 
Then $A(R_5)=\{1, \tau\}$ where $\tau =(1+\sqrt{5})/2$. \\

\noindent
$d=3$: $DS_3(2)$ and there are exactly six optimal two distance sets in $\mathbb{R}^3$ 
(Croft \cite{Croft}, Einhorn-Schoenberg \cite{Ein2}). 
They are the set of vertices of a regular octahedron, a right prism which has a equilateral triangle base and square sides 
and the remaining four sets are subsets of a regular icosahedron.\\

\noindent
$d=4$:  $DS_4(2)=10$ and the optimal two-distance set in $\mathbb{R}^4$ is isomorphic to the set of midpoints of 
the edges of a regular simplex in $\mathbb{R}^4$. This set corresponds to the Petersen graph.\\

\noindent 
$d=5$: $DS_5(2)=16$ and the optimal two-distance set in $\mathbb{R}^5$ is isomorphic to the set given by the Clebsch graph. 
Points of the set are given by the following. 

$$-{e_i}+\sum _{k=1}^5 {e_k} \quad (1\leq i\leq 5),$$
9
$${e_i}+{e_j} \quad (1\leq i<j\leq 5)$$

and the origin $O$ of $\mathbb{R}^5$.\\

\noindent 
$d=6$: $DS_6(2)=27$ and the optimal two-distance set in $\mathbb{R}^6$ is isomorphic to the set obtained from the Schl\"afli graph. \\

\noindent 
$d=7$: $DS_7(2)=29$ and the optimal two-distance set in $\mathbb{R}^7$ is isomorphic to the set which is given by the following points. 

$$-{e_i}+ \frac{1}{7}(3+\sqrt{2})\sum_{k=1}^{7}{e_k} \quad (1\leq i\leq 7),$$ 
$${e_i}+{e_j} \quad (1\leq i<j\leq 7)$$ 
and
$$\frac{1}{7}(2+3\sqrt{2})\sum_{k=1}^7{e_k}.$$

\noindent 
$d=8$: A two-distance set in $\mathbb{R}^8$ with $\binom{10}{2}=45$ points is known. 
Let 
$$X_1=\{e_i -\frac{1}{12}\sum _{k=1}^8e_k | i=1,2,\ldots 8\}\cup \{-\frac{1}{3}\sum _{k=1}^8e_k\}$$ 
and
$$X_2=\{-(x+y) |  x, y\in X_1, x\ne y \}$$ \\
Then $X_1$ is the vertex set of a regular simplex and $X_1\cup X_2$ is a two-distance set with $A(X_1\cup X_2)=\{\sqrt{2}, 2\}$\\

\noindent 
{\it Spherical cases}\quad For $2\leq d\leq 6$, every optimal two-distance set in $\mathbb{R}^d$ is on a sphere. 
Optimal two-distance sets in $S^6$ are given from three Chang graphs or the set of midpoints of edges of a regular simplex in $\mathbb{R}^7$. 
Moreover, Musin \cite{Musin} determined $DS^\ast _d(2)$ for $7\leq d< 40$. 

\begin{thm}\label{spherical 2ds}
$DS^\ast _d(2)=d(d+1)/2$ for the cases where $7\leq d\leq 21,24\leq d< 40$.  
When $d=22,23$, $DS^\ast _{22}(2)=275$ and $DS^\ast _{23}(2)=276$ or $277$. 
\end{thm}

\subsection{Optimal locally two-distance sets}\label{sec ans}
\noindent 
{\it Euclidean cases}\quad By using classifications of optimal two-distance sets and Theorem \ref{max}, 
we have the following proposition. 

\begin{prop}\label{2dr}
Every optimal locally two-distance set in $\mathbb{R}^d$ is a two-distance set for $d=2, 4, 5, 6, 8$.  
Moreover there are four seven-point locally two-distance set in $\mathbb{R}^3$ up to isomorphism and 
five 29-point locally two-distance set in $\mathbb{R}^7$ up to isomorphism. 
In particular $DS_d(2)=LDS_d(2)$ for $d=1, 2, 4\leq d\leq 8$ and $LDS_3(2)=7$. 
\end{prop} 

\begin{proof}
$d=1$: It is clear that every three-point set in $\mathbb{R}^1$ which is not a one-distance set is a locally two-distance set 
and that there is no four-point locally two-distance set in $\mathbb{R}^1$. \\

For $2\leq d\leq 7$, we classify optimal locally two-distance sets in $\mathbb{R}^d$. 
For each case, we pick a saturated subset $Y$ of $X$ and we let $Y'=X\setminus Y$. 
Note that if $X$ is not a two-distance set, then $1\leq \dim(Y)\leq d-1$.\\

\noindent 
$d=2$: We will classify five-point locally two-distance sets $X$ in $\mathbb{R}^2$. 
We may assume that $\dim(Y)=1$ and $|Y|=2$, otherwise $X$ is a two-distance set. 
Let $Y=\{y_1, y_2\}$, $Y'=\{x_1, x_2, x_3\}$ and $A(y_1)=A(y_2)=\{\alpha, \beta \}$. 
Without of generality, we may assume $d(x_1,y_i)=d(x_2,y_i)=\alpha $ and $d(x_3, y_i)=\beta$ for $i\in \{1,2\}$ 
since there are exactly four possibilities for the $x_j$. 
If $d(x_1,x_3)\in \{\alpha ,\beta\}$, then $A(x_1)=\{\alpha , \beta \}$ or $A(x_3)=\{\alpha , \beta \}$. 
This is a contradiction to the maximality of the saturated subset $Y$. So $d(x_1,x_3)=\gamma \notin \{\alpha ,\beta\}$. 
Similarly $d(x_2,x_3)=\gamma$. Therefore $x_3$ is a midpoint of both the segment $y_1y_2$ and the segment $x_1x_2$. 
It is easy to check that such a locally two-distance set does not exist. 
Therefore $\dim(Y)\ne 1$ and $X$ is a two-distance set. By the classification of five-point two-distance sets in $\mathbb{R}^2$, $X=R_5$. \\

\noindent 
$d=3$: We will classify seven-point locally two-distance sets $X$ in $\mathbb{R}^3$. 
We may assume $1\leq \dim(Y)\leq 2$, otherwise $X$ is a two-distance set. We need to consider  two cases (a) $\dim(Y)=1$ and (b) $\dim(Y)=2$. \\
(a) In this case, $|Y|=2$ and $Y'=R_5$ by the above classification. Let $Y=\{y_1, y_2\}$ and $Y'=\{x_1, x_2, \ldots ,x_5\}$. 
Then $d(x_j, y_i)=1$ for any $j\in \{1,2\}$ and $i\in \{1,2,\ldots , 5\}$ or $d(x_j, y_i)=\tau$ for any $j\in \{1,2\}$ and $i\in \{1,2,\ldots , 5\}$. 
In this case, there are two seven-point locally two-distance sets up to isomorphism. \\
(b) In this case, $|Y|\in \{4,5\}$. 
If $|Y|=4$, then $|Y'|=3$. Similar to the case where $d=2$, there exists a point $x\in Y'$ which is the midpoint of the other two points. 
Then $Y\cup \{x\}$ is a five-point locally two-distance set in $\mathbb{R}^2$ and $x$ is a center of 
the circle passing through other four points. By the classification of five-point locally two-distance sets in $\mathbb{R}^2$, 
such a locally two-distance set does not exist. 
If $|Y|=5$, then $|Y'|=2$. In this case, $Y=R_5$ and there are four locally two-distance sets up to isomorphism. 
These sets contains the sets in case (a). \\

\noindent 
$d=4$:  We will classify ten-point locally two-distance sets $X$ in $\mathbb{R}^4$. 
If $\dim(Y)\ne 2$, then $X$ is a two-distance set or $|X|<10$. Therefore we assume $\dim(Y)=2$. 
Then $|Y|=|Y'|=5$ and both $Y$ and $Y'$ are sets of vertices of a regular pentagon. 
Let 
$$Y=\{( \cos \frac{2\pi j}{5},\sin \frac{2\pi j}{5}, 0, 0) |j=0,1,\ldots 4\}$$ and  
$$Y'=\{(0, 0, r\cos \frac{2\pi j}{5}, r\sin \frac{2\pi j}{5})|j=0,1,\ldots 4\}.$$  
Then $d(x,y)=\sqrt{1+r^2}>1$ for any $y\in Y$ and $x\in Y'$. 
Therefore we may assume $d(x,y)=\tau $ where $\tau =(1+\sqrt{5})/2$. 
Then $r=\sqrt{\tau}$ and $A(x)=\{\tau ^{1/2},\tau ,\tau ^{3/2}\}$ for $x\in Y'$. 
This is not a locally two-distance set. 
Therefore a ten-point locally two-distance set is a two-distance set. \\
  
\noindent 
$d=5$: We will classify sixteen-point locally two-distance sets $X$ in $\mathbb{R}^5$. 
Since $DS^\ast _i(2)+LDS_{d-i}(2)<16$ for $1\leq i\leq 4$, $X$ is a two-distance set. \\

\noindent 
$d=6$: We will classify $27$-point locally two-distance sets $X$ in $\mathbb{R}^6$. 
By Corollary \ref{width}, every $27$-point locally two-distance set in $\mathbb{R}^6$ 
is a two-distance set. \\

\noindent 
$d=7$: We will classify $29$-point locally two-distance sets $X$ in $\mathbb{R}^7$. 
If $\dim(Y)\notin \{1,6\}$, then $X$ is a two-distance set or $|X|<29$. 
We divide into two cases: (a) $\dim(Y)=1$ and (b) $\dim(Y)=6$. \\
(a) In this case, similar to the classification of case (a) for $d=3$, 
we prove that there are two $29$-point locally two-distance sets up to isomorphism. \\
(b) In this case, similar to the classification of case (b) for $d=3$, 
we can prove that there are four locally two-distance sets which contain the sets in  case (a). \\

\noindent 
$d=8$: We will consider $45$-point locally two-distance sets in $\mathbb{R}^8$.  
By Corollary \ref{width}, every $45$-point locally two-distance set in $\mathbb{R}^8$ 
is a two-distance set. 
\end{proof} 

\noindent 
{\it Spherical cases}\quad For spherical cases, we have the following proposition by Theorem \ref{partial ans} and Theorem \ref{spherical 2ds}. 

\begin{prop}
$LDS^\ast _d(2)=DS^\ast _d(2)$ for $2\leq d <40$ and $d\notin \{3, 7, 23\}$. 
When $d\in \{3, 7, 23\}$, $LDS^\ast _3(2)=7$, $LDS^\ast _7(2)=29$ and $LDS^\ast _{23}(2)=277$.
In particular, there is a unique optimal locally two-distance set in $S^{d-1}$ if $d\in \{3,7\}$ and 
there is a unique optimal locally two-distance set in $S^{23}$ if $DS^\ast _{23}(2)=276$. 
\end{prop}

\subsection{Optimal locally three-distance sets}
It seems difficult to determine $LDS_d(k)$ and classify the optimal configurations for $k\geq 3$. 
However there is a result for $k=3$ and $d=2$ by Erd\H{o}s-Fishburn \cite{Erd3} and Fishburn \cite{Fish1}. 

\begin{center}
\unitlength 0.1in
\begin{picture}( 15.7700, 15.9600)(  3.3300,-16.9600)
%
\special{pn 8}%
\special{pa 834 620}%
\special{pa 1400 620}%
\special{pa 1400 1186}%
\special{pa 834 1186}%
\special{pa 834 620}%
\special{fp}%
%
\special{pn 13}%
\special{sh 0.600}%
\special{ar 1118 1678 20 20  0.0000000 6.2831853}%
%
\special{pn 13}%
\special{sh 0.600}%
\special{ar 1892 902 20 20  0.0000000 6.2831853}%
%
\special{pn 13}%
\special{sh 0.600}%
\special{ar 1118 120 20 20  0.0000000 6.2831853}%
%
\special{pn 13}%
\special{sh 0.600}%
\special{ar 834 620 20 20  0.0000000 6.2831853}%
%
\special{pn 13}%
\special{sh 0.600}%
\special{ar 1400 620 20 20  0.0000000 6.2831853}%
%
\special{pn 13}%
\special{sh 0.600}%
\special{ar 1400 1186 20 20  0.0000000 6.2831853}%
%
\special{pn 13}%
\special{sh 0.600}%
\special{ar 834 1186 20 20  0.0000000 6.2831853}%
%
\special{pn 8}%
\special{pa 834 620}%
\special{pa 834 1186}%
\special{pa 344 902}%
\special{pa 834 620}%
\special{fp}%
%
\special{pn 8}%
\special{pa 1400 1186}%
\special{pa 1398 620}%
\special{pa 1890 900}%
\special{pa 1400 1186}%
\special{fp}%
%
\special{pn 8}%
\special{pa 1118 1676}%
\special{pa 834 1186}%
\special{pa 1400 1186}%
\special{pa 1118 1676}%
\special{fp}%
%
\special{pn 8}%
\special{pa 1400 620}%
\special{pa 834 620}%
\special{pa 1118 128}%
\special{pa 1400 620}%
\special{fp}%
%
\special{pn 13}%
\special{sh 0.600}%
\special{ar 352 900 20 20  0.0000000 6.2831853}%
\end{picture}%

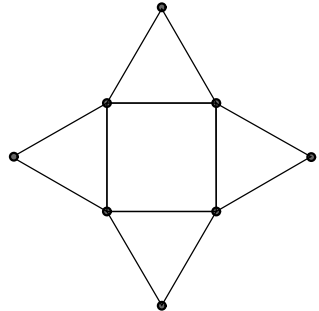
Figure 1.
\end{center}

\begin{prop}\label{fact}
$(\mathrm{i})$ Let $X$ be an eight-point planar set. Then 
$\sum _{P\in X}|A_X(P)|\geq 24$.\\
$(\mathrm{ii})$ Every eight-point planar set $X$ with $\sum _{P\in X}|A_X(P)|=24$ is similar to Figure 1.\\
$(\mathrm{iii})$ Every eight-point locally three-distance set in $\mathbb{R}^2$ is similar to Figure 1.
In particular, $LDS_3(3)=8$. 
\end{prop}
\begin{proof}
(i), (ii) See \cite{Erd3}, \cite{Fish1}.\\
(iii) This is immediate from (i), (ii). 
\end{proof}

The second author proved that $DS_3(3)=12$ and that every twelve-point three-distance set in  $\mathbb{R}^3$ 
is similar to the set of vertices of a regular icosahedron (\cite{Shino3}). 

\begin{prob}
Is every locally three-distance set in $\mathbb{R}^3$ with twelve points similar to the set of vertices of a regular icosahedron?
\end{prob}

In fact, there are many differences between $k$-distance sets and locally $k$-distance sets when cardinalities are small. 
Moreover we saw that $DS_d(k)<LDS_d(k)$ for some cases. 
However no known optimal $k$-distance sets are locally $(k-1)$-distance sets.

\begin{prob}
Are there any optimal $k$-distance sets which are locally $(k-1)$-distance sets?
\end{prob}

\textbf{Acknowledgements.} We thank Oleg Musin. This work is greatly influenced by his paper \cite{Musin}. 
We also thank Eiichi Bannai for his helpful comments. 

\makeatletter
\def\@biblabel#1{#1}
\makeatother


\begin{thebibliography}{99}
\bibitem{Ban}
Ei.\ Bannai and Et.\ Bannai, 
\textit{Algebraic Combinatorics on Spheres}(in Japanese), 
\textit{Springer Tokyo}, 1999.

\bibitem{Bannai-Bannai}
Ei.\ Bannai and Et.\ Bannai, On Euclidean tight $4$-designs, {\it J.\ Math.\ Soc.\ Japan}, \textbf{58} (2006), no. 3, 775--804.

\bibitem{Etsuko2}
Et.\ Bannai, 
On antipodal Euclidean tight $(2e+1)$-designs.  
{\it J.\ Algebraic Combin.} \textbf{24} (2006), no.\ 4, 391--414. 

\bibitem{Ban2}Ei.\ Bannai, Et.\ Bannai, and D. Stanton, 
An upper bound for the cardinality of an s-distance subset 
in real Euclidean space, II, 
\textit{Combinatorica} \textbf{3} (1983), 147--152.

\bibitem{Bannai-Damerell1}
Ei.\ Bannai and R.\ M.\ Damerell, Tight spherical designs. I, {\it J.\ Math.\ Soc.\ Japan}, {\bf 31} (1979), no.\ 1, 199--207.

\bibitem{Bannai-Damerell2}
Ei.\ Bannai and R.\ M.\ Damerell, Tight spherical designs. II, {\it J.\ London Math.\ Soc.\ }(2) {\bf 21} (1980), no.\ 1, 13--30. 

\bibitem{Bannai-Munemasa-Venkov}
Ei.\ Bannai, A.\ Munemasa, and B.\ Venkov, The nonexistence of certain tight spherical designs. With an appendix by Y.-F. S. Petermann, {\it Algebra i Analiz} {\bf 16} (2004), no.\ 4, 1--23; translation in {\it St.\ Petersburg Math.\ J.\ } 16 (2005), no.\ 4, 609--625 

\bibitem{Blo}A.\ Blokhuis, \textit{Few-distance sets}, 
Ph.\ D.\ thesis, Eindhoven Univ.\ of Technology (1983), 
(CWI Tract (7) 1984).

\bibitem{Croft}H.\ T.\ Croft, 
$9$-point and $7$-point configuration in $3$-space, 
\textit{Proc.\ London.\ Math.\ Soc.\ }(3), 
\textbf{12} (1962), 400--424. 

\bibitem{Del2} P.\ Delsarte, 
Four fundamental parameters of a code and their combinatorial significance, 
\textit{Information and Control}, 
\textbf{23} (1973), 407--438 

\bibitem{Del}P.\ Delsarte, J.\ M.\ Goethals, and J.\ J.\ Seidel, 
Spherical codes and designs, 
\textit{Geom.\ Dedicata}, 
\textbf{6} (1977), 363--388

\bibitem{Delsarte-Seidel}
P.\ Delsarte and J.\ J.\ Seidel, Fisher type inequalities for Euclidean $t$-designs, {\it Lin.\ Algebra and its Appl.} 114/115 (1989), 213--230.

\bibitem{Ein1}S.\ J.\ Einhorn and I.\ J.\ Schoenberg, 
On Euclidean sets having only two distances between points I, 
\textit{Nederl Akad.\ Wetensch.\ Proc. Ser. A69=Indag.\ Math.\ }
\textbf{28} (1966), 479--488.

\bibitem{Ein2}S.\ J.\ Einhorn and I.\ J.\ Schoenberg, 
On Euclidean sets having only two distances between points II, 
\textit{Nederl Akad.\ Wetensch.\ Proc. Ser. A69=Indag.\ Math.\ }
\textbf{28} (1966), 489--504.

\bibitem{Erd2}P.\ Erd\H{o}s and P.\ Fishburn, 
Maximum planar sets that determine $k$ distances,
\textit{Discrete Math.\ }, \textbf{160} (1996), 115--125.

\bibitem{Erd3}P.\ Erd\H{o}s and P.\ Fishburn, 
Distinct distances in finite planar sets, 
\textit{Discrete Math.\ } \textbf{175} (1997), 97--132.

\bibitem{Fish1}P.\ Fishburn, 
Convex nonagons with five intervertex distance, 
\textit{Discrete Math.\ } \textbf{252} (2002), 103--122.

\bibitem{Kelly}L.\ M.\ Kelly, 
Elementary Problems and Solutions. Isosceles $n$-points, 
\textit{Amer.\ Math.\ Monthly}, 
\textbf{54} (1947), 227--229.

\bibitem{Lar}D.\ G.\ Larman, C.\ A.\ Rogers, and J.\ J.\ Seidel, 
On two-distance sets in Euclidean space, 
\textit{Bull.\ London Math.\ Soc.\ },
\textbf{9} (1977), 261--267.

\bibitem{Liso}P.\ Lison\v{e}k, 
New maximal two-distance sets, 
\textit{J. Comb.\ Theory, Ser.\ A77} (1997), 318--338.

\bibitem{Musin}
O.R.\ Musin, On spherical two-distance sets, {\it J.\ Combin.\ Theory Ser.\ A} 116 (4) (2009) 988--995. 

\bibitem{Shino1}M.\ Shinohara, 
Classification of three-distance sets in two dimensional Euclidean space, 
\textit{Europ.\ J.\  Combinatorics}, \textbf{25} (2004) 1039--1058.

\bibitem{Shino2}M.\ Shinohara, 
Uniqueness of maximum planar five-distance sets, {\it Discrete Math.\ }{\bf 308} (2008), 3048--3055.

\bibitem{Shino3}M.\ Shinohara, 
Uniqueness of maximum three-distance sets in the three-dimensional Euclidean Space, 
preprint. 

\bibitem{Taylor}
M.\ A.\ Taylor, Cubature for the sphere and the discrete spherical harmonic transform, {\it SIAM J.\ Numer.\ Math. } {\bf 32} (1995), 667--670 
\end{thebibliography}
\end{document}